\documentclass[reqno]{amsart}

\usepackage{alphalph}

\usepackage{stmaryrd}
\usepackage{amsmath, amssymb, amsthm, epsfig}
\usepackage{hyperref, latexsym}
\usepackage{url}
\usepackage[mathscr]{euscript}

\usepackage{color}
\usepackage{fullpage} 
\usepackage{setspace}

\def\today{\ifcase\month\or
  January\or February\or March\or April\or May\or June\or
  July\or August\or September\or October\or November\or December\fi
  \space\number\day, \number\year}

 \newtheorem{theorem}{Theorem}
  
 \newtheorem{lemma}[theorem]{Lemma}
 \newtheorem{proposition}[theorem]{Proposition}
 \newtheorem{corollary}[theorem]{Corollary}
 \theoremstyle{definition}

 \theoremstyle{remark}

 \newcommand{\C}{\mathbb{C}}
 \newcommand{\R}{\mathbb{R}}

 \newcommand{\hh}{\tfrac12}

  \renewcommand{\d}{\text{\rm d}}

\newcommand{\im}{{\rm Im}\,}
\newcommand{\re}{{\rm Re}\,}

\begin{document}
\title[The second moment of $S_n(t)$ \\ on the Riemann hypothesis]{The second moment of $S_n(t)$ \\on the Riemann hypothesis}
\author[Chirre]{Andr\'{e}s Chirre}
\author[Quesada-Herrera]{Oscar E. Quesada-Herrera}
\address{Department of Mathematical Sciences, Norwegian University of Science and Technology, NO-7491 Trondheim, Norway}
\email{carlos.a.c.chavez@ntnu.no }
\address{IMPA - Instituto Nacional de Matem\'{a}tica Pura e Aplicada - Estrada Dona Castorina, 110, Rio de Janeiro, RJ, Brazil 22460-320}
\email{oscarqh@impa.br}

\allowdisplaybreaks
\numberwithin{equation}{section}

\maketitle

\begin{abstract}
	Let $S(t) = \tfrac{1}{\pi} \arg \zeta \big(\hh + it \big)$ be the argument of the Riemann zeta-function at the point $\tfrac12 + it$. For $n \geq 1$ and $t>0$ define its antiderivatives as
	\begin{equation*}
		S_n(t) = \int_0^t S_{n-1}(\tau) \,\d\tau\, + \delta_n\,, 
	\end{equation*}
	where $\delta_n$ is a specific constant depending on $n$ and $S_0(t) := S(t)$. In 1925, J. E. Littlewood proved, under the Riemann Hypothesis, that 
		$$	\int_{0}^{T}|S_n(t)|^2\,\d t = O(T),
$$
	for $n\geq 1$. In 1946, Selberg unconditionally established the explicit asymptotic formulas for the second moments of $S(t)$ and $S_1(t)$. This was extended by Fujii for $S_n(t)$, when $n\geq 2$. Assuming the Riemann Hypothesis, we give the explicit asymptotic formula for the second moment of $S_n(t)$ up to the second-order term, for $n\geq 1$. Our result conditionally refines Selberg's and Fujii's formulas and extends previous work by Goldston in $1987$, where the case $n=0$ was considered.
\end{abstract}

\section{Introduction}
\subsection{Background} Let $\zeta(s)$ be the Riemann zeta-function. Let $N(t)$ denote the number of zeros $\rho = \beta + i \gamma$ of $\zeta(s)$ (counted with multiplicity), such that $0 < \gamma \leq t$, where the zeros with ordinate $\gamma = t$ are counted with weight $\hh$. The classical Riemann von-Mangol\d t formula states that 
\begin{align} \label{1_1}
N(t) = \frac{t}{2\pi} \log \frac{t}{2\pi} - \frac{t}{2\pi} + \frac{7}{8}  + S(t) + O\left( \frac{1}{t}\right),
\end{align}
where $S(t)$ is known as the argument function of the Riemann zeta-function. When $t$ is not the ordinate of a zero of $\zeta(s)$, we define
$$S(t) = \tfrac{1}{\pi} \arg \zeta \big(\hh + it \big),$$
where the argument is obtained by a continuous variation along straight line segments joining the points $2$, $2+i t$ and $\hh + it$, with the convention that $\arg \zeta(2) = 0$. If $t$ is the ordinate of a zero of $\zeta(s)$, we define
$$S(t) = \tfrac{1}{2}\, \lim_{\varepsilon \to 0} \left\{ S(t + \varepsilon) + S(t - \varepsilon)\right\}.$$
To understand the distribution of the zeros of $\zeta(s)$, 
the formula \eqref{1_1} has led to studying the oscillatory character of $S(t)$. J. E. Littlewood \cite{L, L2} and A. Selberg \cite{S1, S2} investigated the behavior and the power of the cancellation in $S(t)$ using its antiderivatives $S_n(t)$. Setting $S_0(t) = S(t)$ we define, for $n\geq 1$ an integer and $t >0$,
\begin{equation*}
S_n(t) = \int_0^t S_{n-1}(\tau) \,\d\tau\, + \delta_n\,,
\end{equation*}
where $\delta_n$ is a specific constant depending on $n$. These are given by
$$\delta_{2k-1} =\frac{ (-1)^{k-1}}{\pi} \int_{1/2}^{\infty} \int_{\sigma_{2k-2}}^{\infty} \ldots \int_{\sigma_{2}}^{\infty} \int_{\sigma_{1}}^{\infty} \log |\zeta(\sigma_0)|\, \d\sigma_0\,\d\sigma_1\,\ldots \,\d \sigma_{2k-2} $$
for $n = 2k-1$, with $k\geq 1$, and
$$\delta_{2k} = (-1)^{k-1} \int_{1/2}^{1} \int_{\sigma_{2k-1}}^{1} \ldots \int_{\sigma_{2}}^{1} \int_{\sigma_{1}}^{1} \d\sigma_0\,\d\sigma_1\,\ldots \,\d \sigma_{2k-1}  = \frac{(-1)^{k-1}}{(2k)! \cdot 2^{2k}}$$
for $n = 2k$, with $k\geq 1$.

\smallskip

Let us recall some estimates for $S_n(t)$. In $1924$, assuming the Riemann Hypothesis (RH), Littlewood \cite[Theorem 11]{L} established the bounds 
\begin{align} \label{1_12}
S_n(t)=O_n\bigg(\dfrac{\log t}{(\log\log t)^{n+1}}\bigg),
\end{align} 
for $n\geq0$. The order of magnitude in \eqref{1_12} has never been improved, and efforts have thus been concentrated in optimizing the values of the implicit constants. The best known versions of these results are due to Carneiro, Chandee and Milinovich \cite{CCM} for $n=0$ and $n=1$, and Carneiro and Chirre \cite{CChi} for $n \geq 2$ (see also \cite[Theorem 2]{CChiM} for a refinement in the error term). In the other direction, Selberg \cite{S1} and Littlewood \cite{L2} first studied the largest positive and negative values of $S_n(t)$. These have also been the subject of recent research, with improvements for $S(t)$ and $S_1(t)$ in the work of Bondarenko and Seip \cite{BS}. Further refinements for $S_n(t)$ were obtained by Chirre and Mahatab \cite{ChiM2} (see also \cite{Chi} and \cite{ChiM}).
\subsection{The second moment for $S_n(t)$} The next step to understand the behavior of the function $S_n(t)$ is to obtain an asymptotic formula for its moments. In this paper we will concentrate on the second moment.

 In $1925$, assuming RH, Littlewood \cite[Theorem 9]{L2} proved for $n\geq 1$ that 
\begin{align} \label{3}  
\int_{0}^{T}|S_n(t)|^2\,\d t = O(T).
\end{align}
A few years later, in $1928$, Titchmarsh \cite[Theorem II]{Tit2} gave the first explicit version of the above result, for $n=1$, establishing that
\begin{align*} 
\int_0^T |S_1(t)|^2 \d t \sim \dfrac{C_1}{2\pi^2}\,T,
\end{align*}
where 
$$
C_1=\sum_{m=2}^\infty \frac{\Lambda^2(m)}{m\left(\log m\right)^{4}}.
$$
Here, $\Lambda(m)$ is the von-Mangol\d t function, which is defined to be $\log p$ if $m=p^k$ (for some prime number $p$ and integer $k\geq 1$), and zero otherwise. Unconditionally, in $1946$ Selberg \cite[Theorems 6 and 7]{S2} established that
\begin{align}  \label{S}
\int_{0}^{T}|S(t)|^2\,\d t = \dfrac{T}{2\pi^2}\,\log\log T + O(T\sqrt{\log\log T}),
\end{align} 
and\footnote{\,\,\, In \cite{S2}, Selberg actually calculated the second moment for the function $S_1(t)-\delta_1$. His formula can be used to deduce \eqref{S1}, by using the unconditional estimate for $S_2(t)$ given by Fujii \cite[Theorem 2]{F1}.}
\begin{align}  \label{S1}
\int_{0}^{T}|S_1(t)|^2\,\d t =\dfrac{C_1}{2\pi^2}\,T + O\bigg(\dfrac{T}{\log T}
\bigg).
\end{align}
 Assuming RH, Selberg \cite{S1} had proved \eqref{S} with the error term $O(T)$. Going even further, he computed all even moments for $S(t)$ and $S_1(t)$. Using these  even moments for $S(t)$, Ghosh \cite{Gh1, Gh2} obtained the asymptotic behavior for all moments of $|S(t)|^\lambda$, with $\lambda>-1$. Ghosh then used this to establish that $|S(t)|$ is normally distributed around its average order $(\log\log T)^{1/2}$. Furthermore, Fujii \cite{fujii} established, assuming RH,
 \begin{align}  \label{Sn22222}
 	\int_{0}^{T}|S_n(t)|^2\,\d t =\dfrac{C_n}{2\pi^2}\,T + O\bigg(\dfrac{T}{\log T}
 	\bigg),
 \end{align}
for $n\geq2$, where $C_n$ is defined in \eqref{1_00}.
 
 On the other hand, Goldston refined \eqref{S}, assuming RH and using the techniques developed by Montgomery \cite{M} in his work on the pair correlation of the zeros of the Riemann zeta-function. 
We define
\begin{align}  \label{1_4}
F(\alpha)=F(\alpha,T)=\bigg(\dfrac{T}{2\pi}\log T\bigg)^{-1}\displaystyle\sum_{0<\gamma, \gamma'\leq T}T^{\,i\alpha(\gamma-\gamma')}w(\gamma-\gamma'),
\end{align}
for $\alpha\in \R$ and $T\geq 2$, where $w(u)=4/(4+u^2)$. Then, Goldston \cite[Theorem 1]{G1} showed that
	\begin{equation} \label{0_3}
\int_0^T |S(t)|^2 \d t = \frac{T}{2\pi^2}\log\log T + \frac{T}{2\pi^2}\left[\int_1^\infty \frac{F(\alpha)}{\alpha^{2}}\d \alpha+\gamma_0-\sum_{m=2}^\infty\sum_{p}\left(\dfrac{1}{m}-\dfrac{1}{m^2}\right)\dfrac{1}{p^m}\right] + o(T),
\end{equation}
as $T\to \infty$, where $\gamma_0$ is Euler's constant. The error term $o(T)$ in \eqref{0_3} was refined by Chan \cite{Chan1, Chan2}, assuming a quantitative form of the Twin Prime Conjecture and a strong asymptotic formula on pair correlation of zeros of the Riemann zeta function due to Bogomolny and Keating \cite{BK}, respectively.

\medskip

Our main result in this paper is to establish an explicit version of \eqref{3} up to the second-order term, extending the result of Goldston \eqref{0_3} for the cases $n\geq 1$. In particular, we obtain refinements of \eqref{S1} and \eqref{Sn22222}, under RH. Note that our second-order term improves the error terms in \eqref{S1} and \eqref{Sn22222}.

\smallskip

\begin{theorem}\label{thm:main}
	Let $n\ge 1$ be an integer. Assume the Riemann Hypothesis. Then 
	\begin{equation*} 
	\int_0^T |S_n(t)|^2\d t = \dfrac{C_n}{2\pi^2}\, T + \frac{T}{2\pi^2\left(\log T\right)^{2n}}\left[\int_1^\infty \frac{F(\alpha)}{\alpha^{2n+2}}\d \alpha -\frac{1}{2n}\right] + O\bigg(\frac{T\sqrt{\log\log T}}{(\log T)^{2n+1/2}}\bigg),
	\end{equation*}
as $T\to\infty$, where
\begin{equation} \label{1_00}
C_n=\sum_{m=2}^\infty \frac{\Lambda^2(m)}{m\left(\log m\right)^{2n+2}}.
\end{equation}
\end{theorem}

\medskip

Let us analyze the constants that appear on Theorem \ref{thm:main}. We highlight that $C_n\to \infty$ when $n\to \infty$. In fact, the growth of these constants is exponential (see Section \ref{sec:num}), of order $$C_n\sim \frac{1}{2(\log 2)^{2n}}.$$
Table $1$ puts in perspective the constant that appears in front of the first-order term, in the small cases $1\leq n \leq 10$. For the second-order term, by following Goldston's argument using \cite[Lemma A]{G1}, it is straightforward to obtain upper and lower bounds for the integral in the second-order term of Theorem \ref{thm:main}. For any $n\ge 1$ we get\footnote{\,\,\, The constants in \eqref{3_12_9_6} may be slightly improved. However, this is far from the expected behavior suggested by  the Strong Pair Correlation Conjecture \cite{M}, and it seems difficult to obtain anything qualitatively closer.}
	\begin{align} \label{3_12_9_6}
	\frac{2}{3^{2n+3}}-\varepsilon \le \int_1^\infty \frac{F(\alpha)}{\alpha^{2n+2}} \d \alpha \le \frac{8}{3}\,\zeta(2n+2) +\varepsilon,
	\end{align}
	for any $\varepsilon>0$ and $T$ sufficiently large. This implies that the second-order term in Theorem \ref{thm:main} has the growth $T/(\log T)^{2n}$. We highlight that this term has a decreasing order of magnitude as $n$ grows. Furthermore, Montgomery has conjectured that for any fixed number $M>1$, we have $F(\alpha, T)=1+o(1)$ as $T\to\infty$, uniformly for $1\leq \alpha\leq M$. This is known as the Strong Pair Correlation Conjecture. This implies that, as $T\to\infty$
$$
\int_1^\infty \frac{F(\alpha)}{\alpha^{2n+2}}\d \alpha= \dfrac{1}{2n+1} + o(1).
$$

\begin{corollary}
	Let $n\ge 1$ be an integer. Assume the Riemann Hypothesis and the Strong Pair Correlation Conjecture. Then 
	\begin{equation*} 
	\int_0^T |S_n(t)|^2\d t = \dfrac{C_n}{2\pi^2}\, T - \frac{T}{4n(2n+1)\pi^2\left(\log T\right)^{2n}}+o\bigg(\frac{T}{\left(\log T\right)^{2n}}\bigg),
	\end{equation*}
	as $T\to\infty$, where $C_n$ was defined in \eqref{1_00}.
\end{corollary}

\begin{table}
	\begin{center}
		\begin{tabular}{|c|c|c|c|}
			\hline
			$n$ & $C_n/2\pi^2$ & 	$n$ &  $C_n/2\pi^2$  \\
			\hline \hline
			1 & 0.079290... &	6 & 2.064933...  \\ \hline
			2 & 0.124743... & 	7 & 4.290884...  \\ \hline
			3 & 0.239241...  & 8 & 8.925169...\\ \hline
			4 & 0.483838...& 9 & 18.571837... \\ \hline
			5 & 0.996243... & 10& 38.650937... \\ \hline
		\end{tabular}
		\vspace{0.2cm}
		\caption{Values for $1 \leq n \leq 10$.}
	\end{center}
\end{table}

\subsection{Outline of the proof} Our proof follows the ideas developed by Goldston in \cite{G1}, and involves additional technical challenges. In Section $2$, we start by obtaining a new representation formula for $S_n(t)$, for $n\geq 1$, associated with a suitable real-valued function $f_n$. For each $n\geq 1$ define the function $f_n:(0,2)\to\R$ as follows:
\begin{align} \label{6_05_13_53}
f_n(x)=\dfrac{x^{n+1}}{n!}\int_{0}^{\infty}\,y^{n}\dfrac{2\sinh\big(y(1-x)\big)}{(e^{y}+(-1)^{n+1}e^{-y})}\d y.
\end{align}
To get the desired formula for $S_n(t)$, we combine an explicit formula due to Montgomery \cite{M} with an expression for $S_n(t)$ implicit in the work of Fujii \cite{F1} (see also \cite[Lemma 2]{CChi}) that depends on the logarithmic derivative of $\zeta(s)$. Our formula relates $S_n(t)$ to a Dirichlet polynomial over primes involved with the function $f_n$, a sum over the zeros of the Riemann zeta-function, and a few extra terms that depend on the parity of $n$. By squaring and integrating, we obtain an expression for the second moment of $S_n(t)$. Using the asymptotic behavior of each term in this expression, we obtain Theorem \ref{thm:main}. These asymptotic formulas will be obtained in the following sections. We highlight that some of the additional technical difficulties come from controlling both the imaginary and real parts of the logarithm of $\zeta(s)$, which will have repercussions throughout this work.

In Section 3, we analyze the second moment of the sum over the zeros of the Riemann zeta-function. Following Goldston, we use the ideas developed by Montgomery \cite{M} to express sums over pairs of zeros of $\zeta(s)$ in terms of the function $F(\alpha)$ defined in \eqref{1_4}. In Section 4, we analyze the terms associated with the sum over primes. Here, we use an argument of Titchmarsh \cite{Tit2} in the estimate of certain integrals involving $S_n(t)$ with oscillatory functions, which have some peculiarities when $n\geq1$. Combining the terms in an appropriate way and using properties of $f_n$, we can take advantage of a surprising cancellation in our analysis. Finally, in Section 5, we analyze the constants $C_n$ numerically using some estimates of sums with prime numbers that could be of independent interest.
 
\subsection{Notation} The symbols $\ll$, $O(\,\cdot\,)$, and $o(\,\cdot\,)$ are used in the standard way. In Section \ref{sec:num}, to compute explicit constants, we use the notation $\alpha=O^*(\beta)$ to mean that $|\alpha|\leq \beta$. For a function $h\in L^1(\R)$, we define the Fourier transform of $h$ by $$\widehat{h}(\xi)=\int_{-\infty}^\infty h(y)e^{-2\pi i\xi y}\d y.$$

\bigskip

\section{The representation for the second moment of $S_n(t)$} 
\subsection{Representation lemma for $S_n(t)$}
We start by obtaining a new representation for the functions $S_n(t)$ for $n\geq 1$. This representation connects $S_n(t)$ with the zeros of the Riemann zeta-function and the prime numbers.
\begin{lemma}\label{lem:repres}
For each fixed $n\geq 1$ let $f_n:\R\to\R$ be defined as in \eqref{6_05_13_53}. Assume the Riemann Hypothesis. Then, for $t\geq 1$ and $x\geq4$, we have:
\begin{align} \label{13_42}
\begin{split} 
S_n(t)= & \dfrac{1}{\pi n!\,(\log x)^{n}}\displaystyle\sum_{\gamma}\im\{i^{n+2}e^{i(\gamma-t)\log x}\}\int_{0}^{\infty}\dfrac{y^{n+1}}{y^2 + ((\gamma-t)\log x)^2}\dfrac{2}{e^{y}+(-1)^{n+1}e^{-y}}\d y \\
& + \dfrac{1}{\pi}\displaystyle\sum_{2\leq m\leq x}\im\{i^{n}m^{-it}\}\dfrac{\Lambda(m)}{\sqrt{m}(\log m)^{n+1}}\,f_{n}\bigg(\dfrac{\log m}{\log x}\bigg) \\
& + \mu_n\,\dfrac{\im\{i^{n}\}}{\pi(\log x)^{n+1}} \log\dfrac{t}{2\pi}  + O\bigg(\dfrac{\sqrt{x}}{t\,(\log x)^{n+2}}\bigg),
\end{split}
\end{align} 
where the first sum runs over the ordinates of the non-trivial zeros of $\zeta(s)$, and $\mu_n=2^{-n-1}(1-2^{-n})\zeta(n+1)$ when $n$ is odd, and zero otherwise.
\end{lemma}
\begin{proof} Assuming RH, by \cite[Lemma 2]{CChi}, we have for $n\geq 1$ that	\begin{equation}\label{Lem1_eq_1}
	S_{n}(t) = -\frac{1}{\pi} \,\,\im{\left\{\dfrac{i^{n}}{n!}\int_{1/2}^{\infty}{\left(\sigma-1/2\right)^{n}\,\frac{\zeta'}{\zeta}(\sigma+it)}\,\d \sigma\right\}}.
	\end{equation}
Let us analyze the integrand in the above expression. By an explicit formula of Montgomery (see \cite[Eq. (2.1) and p. 155]{G1}), for $x\geq4$ and $s=\sigma+it$ with $\sigma>\hh$ and $t\geq 1$, it follows that
	\begin{align} \label{14_8_1:38am}
	x^{\sigma-1/2}\,\dfrac{\zeta'}{\zeta}(\sigma+it) & -x^{1/2-\sigma}\,\dfrac{\zeta'}{\zeta}(1-\sigma+it) \nonumber \\
	& = (2\sigma-1)\displaystyle\sum_{\gamma}\dfrac{x^{i(\gamma-t)}}{(\sigma-1/2)^2+(\gamma-t)^2} -\displaystyle\sum_{m\leq x}\dfrac{\Lambda(m)}{m^{it}}\bigg(\dfrac{x^{\sigma-1/2}}{m^\sigma}-\dfrac{x^{1/2-\sigma}}{m^{1-\sigma}}\bigg)\\
	& \,\,\,\,\, \,\,\,\,+ x^{1/2-it}\bigg(\dfrac{2\sigma-1}{(\sigma-it)(1-\sigma-it)}\bigg)  + O\bigg(\dfrac{x^{-5/2}(\sigma-1/2)}{t}\bigg). \nonumber 
	\end{align} 
	First, we need a relationship between $\frac{\zeta'}{\zeta}(\sigma+it)$ and $\frac{\zeta'}{\zeta}(1-\sigma+it)$ in the above formula. Using the functional equation of $\zeta(s)$ in the form $\zeta(1-s)=\pi^{-s}2^{1-s}\cos(\pi s/2)\,\Gamma(s)\zeta(s)$, the reflection principle, Stirling's formula and the bound $|\re\{\tan s\}|\ll e^{-2\,\im{s}}$ for $|\im{s}|\geq 1$, we obtain for $t\geq 1$ and $\sigma>\hh$:
	\begin{align} \label{14_8_1:36am1}
	\re\dfrac{\zeta'}{\zeta}(1-\sigma+it)=-\re\dfrac{\zeta'}{\zeta}(\sigma+it) - \log\dfrac{t}{2\pi} + O\bigg(\dfrac{\sigma^2}{t}\bigg).
	\end{align}
	By \cite[Eq. (2.3)]{G1} we also get
	\begin{align} \label{14_8_1:36am}
	\im\dfrac{\zeta'}{\zeta}(1-\sigma+it)=\im\dfrac{\zeta'}{\zeta}(\sigma+it) + O\bigg(\dfrac{\sigma-1/2}{t}\bigg).
	\end{align}
Then, combining \eqref{14_8_1:36am1} and \eqref{14_8_1:36am}, we obtain
$$
\dfrac{\zeta'}{\zeta}(1-\sigma+it)=-\overline{\dfrac{\zeta'}{\zeta}(\sigma+it)} - \log\dfrac{t}{2\pi} + O\bigg(\dfrac{\sigma^2}{t}\bigg).
$$
Inserting it  into \eqref{14_8_1:38am} and ordering conveniently, one can see that
\begin{align} 
	\begin{split}
\Big(x^{\sigma-1/2} +(-1)^{n+1}&x^{1/2-\sigma}\Big)\dfrac{\zeta'}{\zeta}(\sigma+it) \\
& = - \,x^{1/2-\sigma}\bigg((-1)^n\dfrac{\zeta'}{\zeta}(\sigma+it) + \overline{\dfrac{\zeta'}{\zeta}(\sigma+it)}\bigg)  \\
& \,\,\,\,\, \,+ (2\sigma-1)\displaystyle\sum_{\gamma}\dfrac{x^{i(\gamma-t)}}{(\sigma-1/2)^2+(\gamma-t)^2}  -\displaystyle\sum_{m\leq x}\dfrac{\Lambda(m)}{m^{it}}\bigg(\dfrac{x^{\sigma-1/2}}{m^\sigma}-\dfrac{x^{1/2-\sigma}}{m^{1-\sigma}}\bigg)\\
& \,\,\,\,\, \,- x^{1/2-\sigma}\log\dfrac{t}{2\pi} + x^{1/2-it}\bigg(\dfrac{2\sigma-1}{(\sigma-it)(1-\sigma-it)}\bigg) + O\bigg(\dfrac{\sigma^2}{t}\big(x^{-5/2}+x^{1/2-\sigma}\big)\bigg).
\end{split}
\end{align}
Dividing the above expression by $C_n(\sigma):=x^{\sigma-1/2}+(-1)^{n+1}x^{1/2-\sigma}$ and inserting it into \eqref{Lem1_eq_1}, we get
\begin{align*}
S_{n}(t) & =  \frac{1}{\pi n!} \int_{1/2}^{\infty}{\dfrac{\left(\sigma-1/2\right)^{n}}{C_n(\sigma)}\,\im\bigg\{i^{n}\bigg((-1)^n\dfrac{\zeta'}{\zeta}(\sigma+it) +
	 \overline{\dfrac{\zeta'}{\zeta}(\sigma+it)}\bigg)}x^{1/2-\sigma}\bigg\}\d \sigma\\
 &\,\,\,\,\,\,\,\,- \frac{1}{\pi n!} \int_{1/2}^{\infty}{\dfrac{\left(\sigma-1/2\right)^{n}}{C_n(\sigma)}\,\im\bigg\{i^{n}\,(2\sigma-1)\displaystyle\sum_{\gamma}\dfrac{x^{i(\gamma-t)}}{(\sigma-1/2)^2+(\gamma-t)^2} }\bigg\}\d \sigma \\
 & \,\,\,\,\,\,\,\, + \frac{1}{\pi n!} \int_{1/2}^{\infty}{\dfrac{\left(\sigma-1/2\right)^{n}}{C_n(\sigma)}\,\im\bigg\{i^{n}\,\displaystyle\sum_{m\leq x}\dfrac{\Lambda(m)}{m^{it}}\bigg(\dfrac{x^{\sigma-1/2}}{m^\sigma}-\dfrac{x^{1/2-\sigma}}{m^{1-\sigma}}\bigg)}\bigg\}\d \sigma \\ 
  & \,\,\,\,\,\,\,\, + \frac{1}{\pi n!} \int_{1/2}^{\infty}{\dfrac{\left(\sigma-1/2\right)^{n}}{C_n(\sigma)}\,\im\bigg\{i^{n}\,x^{1/2-\sigma}\log\dfrac{t}{2\pi}}\bigg\}\d \sigma \\
   & \,\,\,\,\,\,\,\, - \frac{1}{\pi n!} \int_{1/2}^{\infty}{\dfrac{\left(\sigma-1/2\right)^{n}}{C_n(\sigma)}\,\im\bigg\{i^{n}\,x^{1/2-it}\bigg(\dfrac{2\sigma-1}{(\sigma-it)(1-\sigma-it)}\bigg)}\bigg\}\d \sigma \\
      & \,\,\,\,\,\,\,\, + O\Bigg(\int_{1/2}^{\infty}{\dfrac{\left(\sigma-1/2\right)^{n}}{C_n(\sigma)}\,\dfrac{\sigma^2}{t}\big(x^{-5/2}+x^{1/2-\sigma}\big)}\,\d \sigma\Bigg) \\
      & = I_{1,n}(x,t)+I_{2,n}(x,t)+I_{3,n}(x,t)+I_{4,n}(x,t)+I_{5,n}(x,t)+O\big(I_{6,n}(x,t)\big).
\end{align*}
We analyze each term in the above expression.\\
\noindent 1. {\it First term}: Using the fact that $\im\{i^n((-1)^nz+
\overline{z})\}=0$, for $z\in\C$ and $n\geq 1$ we get that $I_{1,n}(x,t)=0$. \\
\noindent 2. {\it Second term}: Using Fubini's theorem\footnote{\,\,\, It is justified by the fact that the number of zeros on the interval $[t,t+1]$ is $O(\log t)$.} and the change of variables $y=(\sigma-1/2)\log x$, it follows that
\begin{align*}
I_{2,n}(x,t)&=-\frac{2}{\pi n!}\displaystyle\sum_{\gamma}\im\{i^{n}x^{i(\gamma-t)}\} \int_{1/2}^{\infty}{\dfrac{\left(\sigma-1/2\right)^{n+1}}{(\sigma-1/2)^2+(\gamma-t)^2} }\,\dfrac{1}{x^{\sigma-1/2}+(-1)^{n+1}x^{1/2-\sigma}}\,\d \sigma\\
& = \frac{1}{\pi n!(\log x)^n}\displaystyle\sum_{\gamma}\im\{i^{n+2}e^{i(\gamma-t)\log x}\} \int_{0}^{\infty}{\dfrac{y^{n+1}}{y^2+((\gamma-t)\log x)^2}}\,\dfrac{2}{e^{y}+(-1)^{n+1}e^{-y}}\,\d y.
\end{align*}
\noindent 3. {\it Third term}: Recalling that $f_n(x)$ is defined in \eqref{6_05_13_53}, similar computations give us
\begin{align*}
I_{3,n}(x,t)& =
\frac{1}{\pi}\displaystyle\sum_{m\leq x}\im\{i^{n}m^{-it}\}\,\dfrac{\Lambda(m)}{\sqrt{m}(\log m)^{n+1}}\,f_{n}\bigg(\dfrac{\log m}{\log x}\bigg).
\end{align*}
\noindent 4. {\it Fourth term}: Note that when $n$ is even, we obtain that $I_{4,n}(x,t)=0$. Let us suppose that $n$ is odd. Then \mbox{           \,\,      } $C_n(\sigma)=2\cosh((\sigma-1/2)\log x)$. By a change of variables and \cite[Eq. 3.552-3]{GR} we get that
\begin{align*} 
I_{4,n}(x,t)& =\dfrac{\im\{i^n\}}{2\pi n!}\log\dfrac{t}{2\pi}\int_{1/2}^{\infty}\dfrac{\left(\sigma-1/2\right)^{n}x^{1/2-\sigma}}{\cosh((\sigma-1/2)\log x)}\d \sigma \\
& =\dfrac{\im\{i^n\}}{\pi\,(\log x)^{n+1}}\,{2^{-n-1}(1-2^{-n})\zeta(n+1)}\log\dfrac{t}{2\pi}.
\end{align*}
\noindent 5. {\it Fifth term}: Using the same change of variables,
\begin{align*}
|I_{5,n}(x,t)|& \ll \sqrt{x} \int_{1/2}^{\infty}\dfrac{\left(\sigma-1/2\right)^{n+1}}{|(\sigma-it)(1-\sigma-it)|}\dfrac{1}{x^{\sigma-1/2}+(-1)^{n+1}x^{1/2-\sigma}}\d \sigma \\
& =\dfrac{\sqrt{x}}{(\log x)^{n}}\int_{0}^{\infty}\dfrac{y^{n+1}}{|((1/2-it)\log x)^2-y^2|}\dfrac{1}{e^y+(-1)^{n+1}e^{-y}}\d y \\
& \ll \dfrac{\sqrt{x}}{t\,(\log x)^{n+2}}\int_{0}^{\infty}\dfrac{y^{n+1}}{e^y+(-1)^{n+1}e^{-y}}\d y \ll  \dfrac{\sqrt{x}}{t\,(\log x)^{n+2}}.
\end{align*}
\noindent 6. {\it Sixth term}: As in the previous term, we have
\begin{align*}
|I_{6,n}(x,t)|\ll \dfrac{1}{t\,(\log x)^{n+1}}.
\end{align*}
Combining all the terms, we obtain the desired result.
\end{proof}

\smallskip

Note that in the above lemma we establish the connection between $S_n(t)$ and the function $f_n$. The following lemma summarizes useful information related to the function $f_n$ and a new auxiliary function $g_n$.
\begin{lemma}\label{lem:taylor} Let $n\geq 1$ be an integer and $f_n:(0,2)\to\R$ be the real valued function defined in \eqref{6_05_13_53}. Then, the function $g_n:(0,2)\to\R$ given by
	\begin{align} \label{21_54}
	g_n(x)= \dfrac{1-f_n(x)}{x^{n+1}}
	\end{align}
	satisfies the following properties:
	\begin{enumerate}
		\item [(I)] $g_n$ can be extended to the interval $(-2,2)$, such that $g_n\in C^\infty\big((-2,2)\big)$, and $g^2_n$ is an even function.
		\item [(II)] For $x\in (-2,2)$, the function $g_n$ has the representation
		\begin{align} \label{3_18}
		g_n(x)=\dfrac{1}{n!}\int_{0}^{\infty}e^{-y}y^{n}\bigg(\dfrac{e^{xy}+(-1)^{n+1}e^{-xy}}{e^{y}+(-1)^{n+1}e^{-y}}\bigg)\d y.
		\end{align}
		\item [(III)] In particular, $g_n(0)= 2^{-n}(1-2^{-n})\zeta(n+1)$ when $n$ is odd, and zero otherwise.
	\end{enumerate}
\end{lemma}
\begin{proof} 
	Using the definition of $f_n$, it follows that for $x\in(0,2)$,
	\begin{align*}
	\dfrac{f_{n}(x)}{x^{n+1}}+\dfrac{1}{n!}\int_{0}^{\infty}e^{-y}y^{n}\bigg(\dfrac{e^{xy}+(-1)^{n+1}e^{-xy}}{e^{y}+(-1)^{n+1}e^{-y}}\bigg)\d y & = \dfrac{1}{n!}\int_{0}^{\infty}\,y^{n}\bigg(\dfrac{e^{(1-x)y}+(-1)^{n+1}e^{-(x+1)y}}{e^{y}+(-1)^{n+1}e^{-y}}\bigg)\d y \\
	& =\dfrac{1}{n!}\int_{0}^{\infty}\,y^{n}e^{-xy}\,\d y=\dfrac{1}{x^{n+1}},
	\end{align*}
	where in the last equality we have used \cite[Eq. 3.351-3]{GR}. This implies that 
	\begin{align} \label{0_55}
	g_n(x)=\dfrac{1}{n!}\int_{0}^{\infty}e^{-y}y^{n}\bigg(\dfrac{e^{xy}+(-1)^{n+1}e^{-xy}}{e^{y}+(-1)^{n+1}e^{-y}}\bigg)\d y,
	\end{align} 
	for $x\in(0,2)$. Using dominated convergence one can see that the right-hand side of \eqref{0_55} defines a function in $C^\infty\big((-2,2)\big)$. Then, this representation allows us to extend the function $g_n$ to $(-2,2)$. On the other hand, $g_n(-x)=(-1)^{n+1}g_n(x)$, and this implies that $g^2_n$ is an even function. When $n$ is even, $g_n$ is an odd function and therefore $g_n(0)=0$. When $n$ is odd, using \cite[Eq. 3.552-3]{GR} we get
	$$
	g_n(0)=\dfrac{1}{n!}\int_{0}^{\infty}\dfrac{e^{-y}y^{n}}{\cosh y}\,\d y= 2^{-n}(1-2^{-n})\,\zeta(n+1).
	$$
\end{proof}

\subsection{Proof of Theorem \ref{thm:main}} Lemma \ref{lem:repres} allows us to obtain the second moment of $S_n(t)$ in terms of certain integrals depending of each summand involved in \eqref{13_42}. Let $n\geq 1$ be a fixed integer. Using Lemma \ref{lem:repres}, we have for $t\geq 1$ and $x\geq 4$ that
\begin{align*} 
S_n(t)- I_{3,n}(x,t)- I_{4,n}(x,t)= I_{2,n}(x,t) + O\bigg(\dfrac{\sqrt{x}}{t\,(\log x)^{n+2}}\bigg),
\end{align*} 
where
\begin{align*}
I_{2,n}(x,t)= \frac{1}{\pi n!(\log x)^n}\displaystyle\sum_{\gamma}\im\{i^{n+2}e^{i(\gamma-t)\log x}\} \int_{0}^{\infty}{\dfrac{y^{n+1}}{y^2+((\gamma-t)\log x)^2}}\,\dfrac{2}{e^{y}+(-1)^{n+1}e^{-y}}\,\d y,
\end{align*}
\begin{align*}
I_{3,n}(x,t)=\frac{1}{\pi}\displaystyle\sum_{m\leq x}\im\{i^{n}m^{-it}\}\,\dfrac{\Lambda(m)}{\sqrt{m}(\log m)^{n+1}}\,f_{n}\bigg(\dfrac{\log m}{\log x}\bigg), 
\end{align*}
and
\begin{align*}
I_{4,n}(x,t)=\mu_n\,\dfrac{\im\{i^{n}\}}{\pi(\log x)^{n+1}} \log\dfrac{t}{2\pi}.
\end{align*}
Then, for $T\geq 3$, squaring the above expression and integrating from $1$ to $T$ we obtain
\begin{align} \label{2_42}
\begin{split}
\int_{1}^T|S_n(t)|^2\,\d t & =\int_{1}^T|I_{2,n}(x,t)|^2\,\d t +2\int_{1}^TS_n(t)\,I_{3,n}(x,t)\,\d t - 
\int_{1}^T|I_{3,n}(x,t)|^2\d t\\
& \,\,\,\,\,\,-
\int_{1}^T|I_{4,n}(x,t)|^2\,\d t+2\int_{1}^TS_n(t)\,I_{4,n}(x,t)\,\d t -2\int_{1}^TI_{3,n}(x,t)\,I_{4,n}(x,t)\,\d t\\
& \,\,\,\,\,\, + O\Bigg(\dfrac{\sqrt{x}}{(\log x)^{n+2}}\int_{1}^T\dfrac{|I_{2,n}(x,t)|}{t}\d t\Bigg)+O\bigg(\dfrac{x}{(\log x)^{2n+4}}\bigg).
\end{split}
\end{align}
Using the continuity of $S_{n}(t)$, we get
\begin{align}  \label{1_22}
\int_{1}^T|S_n(t)|^2\,\d t=\int_{0}^T|S_n(t)|^2\,\d t  + O(1).
\end{align}
Now, let us analyze the right-hand side of \eqref{2_42}. Note that $\mu_n=0$ when $n$ is even. Then, we have that
\begin{align}  \label{1_23}
\int_{1}^T|I_{4,n}(x,t)|^2\d t=\mu^2_n\,\dfrac{(\im\{i^{n}\})^2}{\pi^2(\log x)^{2n+2}}\int_{1}^T \log^2\dfrac{t}{2\pi}\d t = \dfrac{\mu_n^2}{\pi^2(\log x)^{2n+2}}\,T\log^2 T + O\bigg(\dfrac{T\log T}{(\log x)^{2n+2}}\bigg).
\end{align}
Furthermore, using the relation $S'_{n+1}(t)=S_n(t)$, the bound $S_n(t)= O(\log t)$ (see \eqref{1_12}), and integration by parts, we obtain
\begin{align}  \label{1_24}
2\int_{1}^TS_n(t)\,I_{4,n}(x,t)\,\d t &=\mu_n\,\dfrac{2\,\im\{i^{n}\}}{\pi(\log x)^{n+1}} \int_{1}^TS'_{n+1}(t)\log\dfrac{t}{2\pi}\d t=O\bigg(\dfrac{\log^2T}{(\log x)^{n+1}}\bigg).
\end{align}
Observe that by (I) from Lemma \ref{lem:taylor}, it is clear that $|f_n(y)|\ll 1$ for $y\in (0,1]$. Then, using the estimate $\Lambda(m)\leq \log m$ and integration by parts, we have 
\begin{align}  \label{1_25}
\begin{split}
\bigg|2\int_{1}^TI_{3,n}(x,t)\,I_{4,n}(x,t)\,\d t
\bigg|& \ll \dfrac{1}{(\log x)^{n+1}}\displaystyle\sum_{m\leq x}\dfrac{\Lambda(m)}{\sqrt{m}(\log m)^{n+1}}\bigg|\int_{1}^T \im\{i^{n}m^{-it}\}\log\dfrac{t}{2\pi}\d t\bigg|\\
& \ll \dfrac{\log T}{(\log x)^{n+1}}\displaystyle\sum_{m\leq x}\dfrac{1}{\sqrt{m}} \ll \dfrac{\sqrt{x}\log T}{(\log x)^{n+1}}.
\end{split}
\end{align}
We estimate the first error term in \eqref{2_42} using Cauchy-Schwarz to get
\begin{align}  \label{1_26}
\begin{split} 
\dfrac{\sqrt{x}}{(\log x)^{n+1}}\int_{1}^T\dfrac{|I_{2,n}(x,t)|}{t}\d t &\leq
\dfrac{\sqrt{x}}{(\log x)^{n+1}}\Bigg(\int_{1}^T|I_{2,n}(x,t)|^2\d t\Bigg)^{1/2}.
\end{split}
\end{align}
Let us define the following integrals:
$$
R_n(x,T)=\int_{1}^T|I_{2,n}(x,t)|^2\d t, \hspace{1.0cm} H_n(x,T)=2\int_{1}^TS_n(t)\,I_{3,n}(x,t)\,\d t, \hspace{0.2cm}
$$
and 
$$ G_n(x,T)=\int_{1}^T|I_{3,n}(x,t)|^2\d t.
$$
Plugging \eqref{1_22}, \eqref{1_23}, \eqref{1_24}, \eqref{1_25} and \eqref{1_26} into \eqref{2_42} gives us
\begin{align} \label{2_39}
\begin{split}
\int_{0}^T|S_n(t)|^2\,\d t & =R_n(x,T) + H_n(x,T) -G_n(x,T)-\dfrac{\mu_n^2}{\pi^2(\log x)^{2n+2}}\,T\log^2 T\\
& \,\,\,\,\,\, + O\left(\dfrac{T\log T}{(\log x)^{2n+2}}\right)+O\bigg(\dfrac{\sqrt{xR_n(x,T)}}{(\log x)^{n+1}}\bigg)+O\bigg(\dfrac{x\log^2T}{(\log x)^{2n+4}}\bigg).
\end{split}
\end{align}
Choosing $x=T^\beta$, for a fixed $0<\beta<\hh$, we get that
\begin{align*}
\int_{0}^T|S_n(t)|^2\,\d t & =R_n(T^\beta,T) + H_n(T^\beta,T) -G_n(T^\beta,T)-\dfrac{\mu_n^2}{\pi^2\beta^{2n+2}}\dfrac{T}{(\log T)^{2n}}\\
& \,\,\,\,\,\, + O\left(\dfrac{T}{(\log T)^{2n+1}}\right)+O\bigg(\dfrac{T^{\beta/2} \sqrt{R_n(T^\beta,T)}}{(\log T)^{n+1}}\bigg).
\end{align*}
We conclude our desired result by using the asymptotic formulas for $R_n(T^\beta,T)$ and $H_n(T^\beta,T) -G_n(T^\beta,T)$ given by Propositions \ref{lem:R} and \ref{lem:gPlusH}  respectively. We remark that by Proposition \ref{lem:R} and \eqref{3_12_9_6}, we can use the bound $R_n(T^\beta,T)=O(T)$ to estimate the error term. \qed

\smallskip

In the following sections, we  will concentrate on obtaining the asymptotic formulas for $R_n(x,T), H_n(x,T)$ and $G_n(x,T)$. Throughout these sections, we will assume that $n\geq 1$ is a given fixed integer. 
\section{Asymptotic formula for $R_n(x,T)$: The sum over the zeros of $\zeta(s)$}

Our objective is to evaluate the mean square of the sum over the zeros of the Riemann zeta-function that appears in \eqref{2_39}. We recall that for $T\geq 3$ and $x\geq 4$,
$$
R_n(x,T)=\dfrac{1}{\pi^2 (n!)^2(\log x)^{2n}}\int_{1}^T\Bigg|\displaystyle\sum_{\gamma}\im\{i^{n+2}e^{i(\gamma-t)\log x}\}\int_{0}^{\infty}\dfrac{y^{n+1}}{y^2 + ((\gamma-t)\log x)^2}\dfrac{2}{(e^{y}+(-1)^{n+1}e^{-y})}\d y\Bigg|^2\d t.
$$

\begin{lemma} \label{17_8_1:31am}
	Let $g_n$ be the function defined in \eqref{21_54}. Assume the Riemann Hypothesis. Then, for $T\geq 3$ and $x\geq 4$ we have
	\begin{align}  \label{20_39}
	R_{n}(x,T)=\dfrac{1}{(\log x)^{2n+1}}\displaystyle\sum_{0<\gamma,\gamma'\leq T}\widehat{k_{n}}((\gamma-\gamma')\log x) + O\bigg(\dfrac{\log^3T}{(\log x)^{2n}}\bigg),
	\end{align}
where the function $k_n:\R\to\R$ is given by
	\begin{equation} \label{18_30}
	k_n(\xi) = \left\{
	\begin{array}{ll}
	g^2_n(2\pi\xi), & \mathrm{if\ } |\xi|\leq \frac{1}{2\pi}\\
	\dfrac{1}{(2\pi\xi)^{2n+2}},    & \mathrm{if\ } |\xi|\geq \frac{1}{2\pi}.
	\end{array}
	\right.
	\end{equation}
Moreover, we have that
\begin{align} \label{3_20}
|\widehat{k_n}(y)|\ll \min\bigg\{1,\dfrac{1}{|y|^2}\bigg\}.
\end{align} 
\end{lemma}
\begin{proof} Define the function
	$$
	h_{n}(u)=\im\{i^{n+2}e^{iu}\} \int_{0}^{\infty}\dfrac{y^{n+1}}{y^2+u^2}\dfrac{2}{(e^{y}+(-1)^{n+1}e^{-y})}\d y.
	$$
Since $|h_{n}(u)|\ll \min\{1,1/{u^2}\}\ll1/(1+u^2)$, using Fubini's theorem we have 
	$$
	R_n(x,T)=\dfrac{1}{\pi^2(n!)^2(\log x)^{2n}}\displaystyle\sum_{\gamma,\gamma}\int_{1}^{T}h_{n}((\gamma-t)\log x)\,h_{n}((\gamma'-t)\log x)\,\d t.
	$$
	Note that $h_n$ is an even function when $n$ is odd and $h_n$ is an odd function when $n$ is even. Using an argument of Montgomery \cite[p. 187]{M} (see also \cite[p. 158]{G1}) one can see that
	\begin{align} \label{19_26}
	\begin{split} 
	R_{n}(x,T)& =\dfrac{1}{\pi^2(n!)^2(\log x)^{2n}}\displaystyle\sum_{0<\gamma,\gamma'\leq T}\int_{-\infty}^{\infty}h_{n}((\gamma-t)\log x)\,h_{n}((\gamma'-t)\log x)\,\d t + O\bigg(\dfrac{\log^3T}{(\log x)^{2n}}\bigg) \\
	& =\dfrac{(-1)^{n+1}}{\pi^2(n!)^2(\log x)^{2n+1}}\displaystyle\sum_{0<\gamma,\gamma'\leq T}h_{n}*h_{n}((\gamma-\gamma')\log x) + O\bigg(\dfrac{\log^3T}{(\log x)^{2n}}\bigg).
	\end{split} 
	\end{align}
	Let us calculate the Fourier transform of $h_{n}$. Using Fubini's theorem, it follows that for $\xi> 0$
	\begin{align*}
	\widehat{h_{n}}(\xi) & =\int_{-\infty}^{\infty}\Bigg(\im\{i^{n+2}e^{iu}\} \int_{0}^{\infty}\dfrac{y^{n+1}}{y^2+u^2}\dfrac{2}{(e^{y}+(-1)^{n+1}e^{-y})}\d y\Bigg)(\cos(2\pi \xi u)-i\sin(2\pi \xi u))\,\d u  \\
		& = \im\{i^{n+2}\}\int_{0}^{\infty}\Bigg( \int_{0}^{\infty}\dfrac{\cos(2\pi \xi u)\cos(u)}{y^2+u^2}\d u\Bigg)\dfrac{4\,y^{n+1}}{(e^{y}+(-1)^{n+1}e^{-y})}\,\d y  \\
	&\,\,\,\,\,\,\,-  i\,\im\{i^{n+3}\}\int_{0}^{\infty}\Bigg( \int_{0}^{\infty}\dfrac{\sin(2\pi \xi u)\sin(u)}{y^2+u^2}\d u\Bigg)\dfrac{4\,y^{n+1}}{(e^{y}+(-1)^{n+1}e^{-y})}\,\d y,
\end{align*}
where we have used the parity of the involved functions.	Then, using the formulas \cite[Eq. 3.742-1 and 3.742-3]{GR} we write
\begin{align*}
\widehat{h_{n}}(\xi) &= \pi\,\im\{i^{n+2}\}\int_{0}^{\infty}\big(e^{-|2\pi\xi-1|y}+e^{-(2\pi\xi+1)y}\big)\dfrac{y^{n}}{(e^{y}+(-1)^{n+1}e^{-y})}\,\d y  \\
&\,\,\,\,\,\,\,-  i\pi\,\im\{i^{n+3}\}\int_{0}^{\infty}\big(e^{-|2\pi\xi-1|y}-e^{-(2\pi\xi+1)y}\big)\dfrac{y^{n}}{(e^{y}+(-1)^{n+1}e^{-y})}\,\d y.
\end{align*}
For $2\pi\xi\geq 1$, making a separate computation of the $n$ odd and $n$ even cases, using \cite[Eq. 3.351-3]{GR}, we obtain
\begin{align*}
\widehat{h_{n}}(\xi) &=\big(\im\{i^{n+2}\}-i\,\im\{i^{n+3}\}\big)\dfrac{n!}{{2^{n+1}\pi^{n}\xi^{n+1}}}.
\end{align*}
On the other hand, for $0<2\pi\xi<1$ we obtain that
\begin{align*}
\widehat{h_{n}}(\xi) &= \pi\,\im\{i^{n+2}\}\int_{0}^{\infty}e^{-y}y^{n}\dfrac{2\cosh(2\pi\xi y)}{(e^{y}+(-1)^{n+1}e^{-y})}\,\d y  -  i\pi\,\im\{i^{n+3}\}\int_{0}^{\infty}e^{-y}y^{n}\dfrac{2\sinh(2\pi\xi y)}{(e^{y}+(-1)^{n+1}e^{-y})}\,\d y.
\end{align*}
Defining the even real-valued function 
\begin{align*}  
k_{n}(\xi)=\frac{(-1)^{n+1}}{\pi^2(n!)^2}(\widehat{h_{n}}(\xi))^2,
\end{align*}
we have
\begin{align*}
\widehat{k_{n}}(y)=\frac{(-1)^{n+1}}{\pi^2(n!)^2}\big(h_{n}*h_{n}\big)(y),
\end{align*}  
and this implies in \eqref{19_26} that 
	\begin{align*} 
R_{n}(x,T)=\dfrac{1}{(\log x)^{2n+1}}\displaystyle\sum_{0<\gamma,\gamma'\leq T}\widehat{k_n}((\gamma-\gamma')\log x) + O\bigg(\dfrac{\log^3T}{(\log x)^{2n}}\bigg).
\end{align*}
Finally, we calculate $k_n$. For $2\pi\xi\geq 1$ we obtain 
\begin{align*}
k_n(\xi)=\frac{(-1)^{n+1}}{\pi^2(n!)^2}\bigg(\big(\im\{i^{n+2}\}-i\,\im\{i^{n+3}\}\big)\dfrac{n!}{{2^{n+1}\pi^{n}\xi^{n+1}}}\bigg)^2 = \dfrac{1}{(2\pi\xi)^{2n+2}},
\end{align*}
and using the parity of the involved functions, we get that the above expression holds for $|2\pi\xi|\geq 1$. On the other hand, for $0< 2\pi\xi< 1$, we have that
\begin{align*}
k_n(\xi)& =\frac{(-1)^{n+1}}{\pi^2(n!)^2}\Bigg(\pi\,\im\{i^{n+2}\}\int_{0}^{\infty}e^{-y}y^{n}\dfrac{2\cosh(2\pi\xi y)}{(e^{y}+(-1)^{n+1}e^{-y})}\,\d y   \\
& \,\,\,\,\,\,\,\,\,\,\,\,\,\,\,\,\,\,\,\,\,\,\,\,\,\,\,\,\,\,\,\,\,\,\,\,\,\,\,\,\,\,\,\,\,\,\,\,\,\,\,\,\,\,\,\,\,\,\,\,\,\,\,\,\,\,\,\,\,\,-  i\pi\,\im\{i^{n+3}\}\int_{0}^{\infty}e^{-y}y^{n}\dfrac{2\sinh(2\pi\xi y)}{(e^{y}+(-1)^{n+1}e^{-y})}\,\d y\Bigg)^2 \\
& = g^2_n(2\pi\xi),
\end{align*}
where in the last line we have treated separately the cases $n$ odd and $n$ even, and used \eqref{3_18}. Using (I) from Lemma \ref{lem:taylor}, it follows that the above expression holds for $|\xi|\leq \frac{1}{2\pi}$. To prove the estimate \eqref{3_20} (see \cite[p. 161]{G1}), we use that $k_n\in L^1(\R)$ implies $|\widehat{k_n}(\xi)|\ll1$, and that integration by parts twice\footnote{\,\,\, The function $k_n$ is absolutely continuous and has bounded derivatives on $\R-\{\pm\frac{1}{2\pi}\}$.} implies $|\widehat{k_n}(y)|\ll \frac{1}{|y|^2}$.
\end{proof}

\medskip

Finally, the following proposition establishes the relation between $R_n(x,T)$ and the function $F(\alpha,T)$.
\begin{proposition} \label{lem:R} Let $0<\beta\leq 1$ be a fixed number. Assume the Riemann Hypothesis. Then, \begin{align*}
R_{n}(T^\beta,T)=\dfrac{T}{2\pi^2(\log T)^{2n}}\Bigg[\bigg(A_n+\dfrac{1}{2n}\bigg)\dfrac{1}{\beta^{2n}}+\bigg(\int_{1}^{\infty}\dfrac{F(\alpha)}{\alpha^{2n+2}}\d\alpha-\dfrac{1}{2n}\bigg)+\dfrac{2\,\mu^2_n}{\beta^{2n+2}}\Bigg] + O\bigg(\dfrac{T\sqrt{\log\log T}}{(\log T)^{2n+1/2}}\bigg),
\end{align*}
as $T\to\infty$, where
\begin{align} \label{22_24}
A_{n}=\int_0^1
\alpha\,g^2_n(\alpha)\,\d \alpha, 
\end{align}
and $\mu_n$ is defined as in Lemma \ref{lem:repres}. 
\end{proposition}
\begin{proof} Let us analyze the main term in \eqref{20_39}. The estimate \eqref{3_20} and a classical argument \cite[p. 161]{G1} imply that
\begin{align} \label{18_22_01_04}
\displaystyle\sum_{0<\gamma,\gamma'\leq T}\widehat{k_{n}}((\gamma-\gamma')\log x) = \displaystyle\sum_{0<\gamma,\gamma'\leq T}\widehat{k_{n}}((\gamma-\gamma')\log x)\,w(\gamma-\gamma')+O(T),
\end{align}
where $w(u)=4/(4+u^2)$. Letting $x=T^\beta$, from \eqref{1_4} one can see that
\begin{align} \label{18_25_01_04}
\displaystyle\sum_{0<\gamma,\gamma'\leq T}\widehat{k_{n}}((\gamma-\gamma')\log x)\,w(\gamma-\gamma')= \dfrac{T\log T}{(2\pi)^2\beta}\int_{-\infty}^{\infty}F(\alpha)\,k_n\bigg(\dfrac{\alpha}{2\pi\beta}\bigg)\d \alpha.
\end{align}
To evaluate the integral on the right-hand side of \eqref{18_25_01_04}, we use the fact that $F(\alpha)$ is even and we split this integral into the intervals $[0,\beta], \, [\beta,1]$ and $[1,\infty)$. Moreover, we calculate these integrals using the asymptotic formula\footnote{\,\,\, This result is due to Goldston and Montgomery \cite[Lemma 8]{GM}, refining the original work of Montgomery \cite{M}.} for $F(\alpha)$: As $T\to\infty$, we have
\begin{align} \label{0_29}
F(\alpha) = \left(\alpha +T^{-2\alpha}\log T\right)(1 + o(1)),
\end{align}
uniformly for $0\leq \alpha \leq 1$, where $o(1)= O\Big(\sqrt{\frac{\log\log T}{\log T}}\Big).$  \\
\noindent 1. {\it  On the interval $[0,\beta]$}: Note that, using (I) from Lemma \ref{lem:taylor}, we have that $g^2_n(\alpha)=g^2_n(0)+O(\alpha^2)$ for $\alpha\in[0,1]$. Then, using \eqref{18_30}, \eqref{0_29} and the fact that $\beta\int_{0}^1T^{-2\beta\alpha}\log T\,\d\alpha=\frac{1}{2}+O\left(\frac{1}{
\log^2T}\right)$, we get
\begin{align*}
\int_0^\beta F(\alpha)& k_{n}\bigg(\dfrac{\alpha}{2\pi\beta}\bigg)\d \alpha = \bigg(\beta^2\int_0^1
\alpha\,g^2_n(\alpha)\,\d \alpha+ \dfrac{g^2_n(0)}{2}+O\bigg(\dfrac{1}{\log ^2T}\bigg)\bigg)(1+o(1)).
\end{align*}
We remark, by (III) from Lemma \ref{lem:taylor}, that $
g^2_n(0)=4\mu^2_n$. \\
\noindent 2. {\it  On the interval $[\beta,1]$}: Here, by \eqref{0_29}, $F(\alpha)=\alpha+o(1)$. Then, we handle this integral using \eqref{18_30} to get
\begin{align*}
\int_{\beta}^{1}F(\alpha)\,k_{n}\bigg(\dfrac{\alpha}{2\pi\beta}\bigg)\d \alpha& =  \int_{\beta}^{1}(\alpha+o(1))\bigg(\dfrac{\beta}{\alpha}\bigg)^{2n+2}\d \alpha= \dfrac{1}{2n}\beta^2 - \dfrac{1}{2n}\beta^{2n+2}+ o(1).
\end{align*}
\noindent 3. {\it  On the interval $[1,\infty)$}: In this case we write
\begin{align*}
\int_{1}^{\infty}F(\alpha)\,k_{n}\bigg(\dfrac{\alpha}{2\pi\beta}\bigg)\d \alpha& =  \int_{1}^{\infty}F(\alpha)\bigg(\dfrac{\beta}{\alpha}\bigg)^{2n+2}\d \alpha= \beta^{2n+2}\int_{1}^{\infty}\dfrac{F(\alpha)}{\alpha^{2n+2}}\,\d \alpha.
\end{align*}
Finally, inserting the above estimates in \eqref{18_25_01_04} and combining with \eqref{18_22_01_04} and Lemma \ref{17_8_1:31am}, we conclude the desired result.
\end{proof}

\bigskip

\section{Asymptotic formulas for $G_n(x,T)$ and $H_n(x,T)$: The sum over the prime numbers}

\subsection{The terms $G_n(x,T)$ and $H_n(x,T)$} We recall that, for $T\geq 3$ and $x\geq 4$, we have defined
\[
G_n(x,T)= \frac{1}{\pi^2}\int_1^T
\Bigg|\displaystyle\sum_{m\leq x}\im\{i^{n}m^{-it}\}\dfrac{\Lambda(m)}{\sqrt{m}(\log m)^{n+1}}\,f_{n}\bigg(\dfrac{\log m}{\log x}\bigg)
\Bigg|^2\d t
\]
and
\begin{align}  \label{14_52}
H_{n}(x,T) =\dfrac{2}{\pi}\displaystyle\sum_{m\leq x}\Bigg(\int_{1}^TS_{n}(t)\,\im\{i^{n}m^{-it}\}\,\d t\Bigg)\dfrac{\Lambda(m)}{\sqrt{m}(\log m )^{n+1}}\,f_{n}\bigg(\dfrac{\log m}{\log x}\bigg).
\end{align} 
We can get the following expression for $G_n(x,T)$ using similar computations as Goldston.
\begin{lemma}\label{lem:dirPol}
	For $T\geq 3$ and $x\geq 4$, we have that
	\begin{equation*}
	G_n(x,T) = \frac{T}{2\pi^2}
	\displaystyle\sum_{m\leq x}\dfrac{\Lambda^2(m)}{m(\log m)^{2n+2}}\,f^2_{n}\bigg(\dfrac{\log m}{\log x}\bigg) + O(x^{2}).
	\end{equation*}
\end{lemma}
\begin{proof}
See \cite[pp. 164-165]{G1}.
\end{proof}

The expression for $H_n(x,T)$ is more subtle, since it requires some modification to the computations of Titchmarsh \cite{Tit2} that arises when $n\ge 1$.

\begin{lemma} \label{lem:intermediate} Assume the Riemann Hypothesis. Then, for $T\geq 3$ and $x\geq 4$, we have 
$$H_{n}(x,T) =\dfrac{T}{\pi^2}\displaystyle\sum_{m\leq x}\dfrac{\Lambda^2(m)}{m(\log m )^{2n+2}}\,f_{n}\bigg(\dfrac{\log m}{\log x}\bigg)+ O(x^2\log T).
$$
\end{lemma}
\begin{proof} First, let us calculate the integral inside of \eqref{14_52}. Using integration by parts in \eqref{Lem1_eq_1}, it follows that, for $t>0$,
	\begin{equation*}
	S_{n}(t) = \frac{1}{\pi} \,\,\im{\left\{\dfrac{i^{n}}{(n-1)!}\int_{1/2}^{\infty}{\left(\sigma-1/2\right)^{n-1}\,\log\zeta(\sigma+it)}\,\d \sigma\right\}}.
	\end{equation*}
Then, using the identity
$$
\im\{i^{n}m^{-it}\}=\dfrac{(-1)^ni^{n+1}}{2}(m^{it}+(-1)^{n+1}m^{-it}),
$$
and Fubini's theorem, we get 
\begin{align}  \label{17_08}
\begin{split} 
\int_{1}^TS_{n}(t)\,& \im\{i^{n}m^{-it}\}\,\d t \\
& =\dfrac{1}{2\pi(n-1)!} \int_{1/2}^{\infty}\left(\sigma-1/2\right)^{n-1}\re\bigg\{\int_{0}^T\log\zeta(\sigma+it)(m^{it}+(-1)^{n+1}m^{-it})\,\d t\bigg\}\,\d \sigma + O(1).
\end{split}
\end{align}
Now, we compute the integral from $0$ to $T$, following the idea in \cite[Lemma $\gamma$]{Tit2}. Let $m\geq 2$ be a natural number and $\hh<\sigma<2$. Consider the integral
$$
\int_{\partial R}\log\zeta(s)\,m^s\d s,
$$
where $R$ is the rectangle with vertices $\sigma$, $2$, $2+iT$ and $\sigma+iT$ with suitable indentations to exclude the point $s=1$. The function $\log\zeta(s)$ is analytic inside the contour $R$, and the radii of $s=1$ may be made to tend to zero. Then, using Cauchy's theorem we have that
\begin{align*}
i\int_{0}^T\log\zeta(\sigma+it)\,m^{\sigma+it}\d t = \int_{\sigma}^2\log\zeta(\alpha)\,m^\alpha \,\d \alpha + i\int_{0}^T\log\zeta(2+it)\,m^{2+it}\d t - \int_{\sigma}^2\log\zeta(\alpha+iT)\,m^{\alpha+iT}\d \alpha.
\end{align*}
Note that
$\int_{\sigma}^2\log\zeta(\alpha)\,m^\alpha \,\d \alpha = O(m^2)$. Then, by \cite[Lemmas $\alpha$ and $\beta$]{Tit2} we get that 
\begin{align} \label{18_26_05_04}
\int_{0}^T\log\zeta(\sigma+it)m^{it}\d t=\dfrac{\Lambda(m)}{m^{\sigma}\log m}T + O(m^{2-\sigma}\log T).
\end{align} 
Similarly, using the integral
$$
\int_{\partial R}\log\zeta(s)\,m^{-s}\d s,
$$
around the same contour, it follows that
\begin{align} \label{18_26_05_05}
\int_{0}^T\log\zeta(\sigma+it)m^{-it}\d t=O(\log T).
\end{align} 
Therefore, combining \eqref{18_26_05_04} and \eqref{18_26_05_05}, we get for $\hh<\sigma<2$ that
\begin{align} \label{21_21_05_05}
\int_{0}^T\log\zeta(\sigma+it)\big(m^{it}+(-1)^{n+1}m^{-it}\big)\d t=\dfrac{\Lambda(m)}{m^{\sigma}\log m}T + O(m^{2-\sigma}\log T).
\end{align}
On the other hand, using the expansion of the logarithm of $\zeta(s)$ and Fubini's theorem, we have for $\sigma\geq 2$ that
\begin{align} \label{21_21_05_06}
\begin{split}
\int_{0}^T\log\zeta(\sigma+it)\big(m^{it}+(-1)^{n+1}m^{-it}\big)\d t & =\displaystyle\sum_{k\geq 2} \dfrac{\Lambda(k)}{k^\sigma\log k}\int_{0}^T\dfrac{m^{it}+(-1)^{n+1}m^{-it}}{k^{it}}\d t\\
& = \dfrac{\Lambda(m)}{m^\sigma\log m}\int_{0}^T\big(1+(-1)^{n+1}m^{-2it}\big)\d t \\
& \,\,\,\,\,\,\, +\displaystyle\sum_{\substack{k\geq 2\\k\neq m}} \dfrac{\Lambda(k)}{k^\sigma\log k}\int_{0}^T\bigg(\bigg(\dfrac{m}{k}\bigg)^{it}+(-1)^{n+1}(mk)^{-it}\bigg)\d t \\
& =\dfrac{\Lambda(m)}{m^\sigma\log m}\,T + O\Bigg(\displaystyle\sum_{k\geq 2}\dfrac{1}{k^\sigma}\Bigg) + O\left(\displaystyle\sum_{\substack{k\geq 2 \\ k\neq m}} \dfrac{1}{k^\sigma|\log(m/k)|}\right)\\
& =\dfrac{\Lambda(m)}{m^\sigma\log m}\,T + O\bigg(\dfrac{1}{2^\sigma}\bigg),
\end{split}
\end{align} 
where in the last sum we have used that  $\sum_{\substack{k\geq 2\\ k\neq m}} \frac{1}{{k^2|\log(m/k)|}}$ is bounded (see \cite[p. 451]{Tit2}). Therefore, inserting \eqref{21_21_05_05} and \eqref{21_21_05_06} in \eqref{17_08} and using \cite[Eq. 3.351-3]{GR} we have
\begin{align*} 
\int_{1}^TS_{n}(t)\,\im\{i^{n}m^{-it}\}\,\d t  =\dfrac{\Lambda(m)\,T}{2\pi\sqrt{m}(\log m)^{n+1}} + O(m^{3/2}\log T).
\end{align*}
Inserting it in \eqref{14_52} we get
\begin{align*}  
H_{n}(x,T) =\dfrac{T}{\pi^2}\displaystyle\sum_{m\leq x}\dfrac{\Lambda^2(m)}{m\,(\log m )^{2n+2}}\,f_{n}\bigg(\dfrac{\log m}{\log x}\bigg)+ O\Bigg(\displaystyle\sum_{m\leq x}\dfrac{m\,\Lambda(m)}{(\log m )^{n+1}}\bigg|f_{n}\bigg(\dfrac{\log m}{\log x}\bigg)\bigg|\log T\Bigg).
\end{align*} 
Finally, using the bound $|f_n(y)|\ll1$ for $y\in(0,1]$ in the error term, we get
\begin{align*}
\Bigg|\displaystyle\sum_{m\leq x}\dfrac{m\,\Lambda(m)}{(\log m )^{n+1}}\,f_{n}\bigg(\dfrac{\log m}{\log x}\bigg)\Bigg| \ll \displaystyle\sum_{m\leq x}m\ll x^2.
\end{align*}
\end{proof}

\subsection{The power of cancelation in $H_n(x,T)-G_n(x,T)$} Here, we will obtain the asymptotic behavior for the difference $H_n(T^\beta,T)-G_n(T^\beta,T)$, as $T\to\infty$. It is possible to obtain asymptotic formulas for $H_n(T^\beta,T)$ and $G_n(T^\beta,T)$ independently, as we did in Proposition \ref{lem:R} for $R_n(T^\beta,T)$. However, the expressions are much more complicated, so we will take advantage of a surprising cancellation in their difference.

\begin{proposition}\label{lem:gPlusH} 
Let $0<\beta<\hh$ be a fixed number. Assume the Riemann Hypothesis. Then,
\begin{equation*}
    H_n(T^\beta,T)-G_n(T^\beta,T)=\frac{T}{2\pi^2}\displaystyle\sum_{m=2}^\infty\dfrac{\Lambda^2(m)}{m(\log m)^{2n+2}}  - \frac{T}{2\pi^2\beta^{2n}\left(\log T\right)^{2n}}\left[A_n+\frac{1}{2n}\right] + O\bigg(\frac{T}{(\log T)^{2n+1}}\bigg),
\end{equation*}
as $T\to\infty$, where $A_n$ is defined as in \eqref{22_24}.
\end{proposition}
\begin{proof}
Using Lemmas \ref{lem:dirPol} and \ref{lem:intermediate}, and completing the square, we get for $x=T^\beta$,
	\begin{align} \label{eq:gPlusH1}
\begin{split}
H_n(T^\beta,T)& -G_n(T^\beta,T) \\
	&  = \frac{T}{2\pi^2}\displaystyle\sum_{m\leq x}\dfrac{\Lambda^2(m)}{m(\log m)^{2n+2}}\bigg[2f_{n}\bigg(\dfrac{\log m}{\log x}\bigg)-f^2_{n}\bigg(\dfrac{\log m}{\log x}\bigg)\bigg] + O\big(T^{2\beta}\log T\big)\\
	& = \frac{T}{2\pi^2}\displaystyle\sum_{m\leq x}\dfrac{\Lambda^2(m)}{m(\log m)^{2n+2}} - \frac{T}{2\pi^2}\displaystyle\sum_{m\leq x}\dfrac{\Lambda^2(m)}{m(\log m)^{2n+2}}\bigg[1-f_{n}\bigg(\dfrac{\log m}{\log x}\bigg)\bigg]^2+ O\big(T^{2\beta}\log T\big)\\
	&= \frac{T}{2\pi^2}\displaystyle\sum_{m\leq x}\dfrac{\Lambda^2(m)}{m(\log m)^{2n+2}} - \frac{T}{2\pi^2(\log x)^{2n+2}}\displaystyle\sum_{m\leq x}\dfrac{\Lambda^2(m)}{m}\,g^2_{n}\bigg(\dfrac{\log m}{\log x}\bigg)+O\big(T^{2\beta}\log T\big).
	\end{split}
	\end{align}
Using Lemma \ref{boundforM} and partial summation, it is clear that	
\begin{align} \label{21_26}
\begin{split}
\displaystyle\sum_{m\leq x}\dfrac{\Lambda^2(m)}{m(\log m)^{2n+2}} =\displaystyle\sum_{m=2}^\infty\dfrac{\Lambda^2(m)}{m(\log m)^{2n+2}} -\frac{1}{2n(\log x)^{2n}} +  O\left(\frac{1}{\sqrt{x}(\log x)^{2n-1}}\right).
\end{split}
\end{align}
Let us analyze the second term in \eqref{eq:gPlusH1}. By the estimate $|g_n(y)|\ll 1$ for $y\in[0,1]$, we get
\begin{align} \label{20_13}
\displaystyle\sum_{m\leq x}\dfrac{\Lambda^2(m)}{m}\,g^2_{n}\bigg(\dfrac{\log m}{\log x}\bigg) = \displaystyle\sum_{p\leq x}\dfrac{\log^2 p}{p}\,g^2_{n}\bigg(\dfrac{\log p}{\log x}\bigg) + O(1).
\end{align}
To analyze the sum over primes on the right-hand side of \eqref{20_13}, we use\footnote{\,\,\, This can be obtained using integration by parts in \cite[Theorem 2.7 (b)]{MV}.}
\[
P(y) = \sum_{p\le y}\frac{\log^2p}{p} = \dfrac{\log^2y}{2}+O(\log y),
\]
for $y\geq 2$. Then, using integration by parts and the bound $|g_n(y)\,g'_n(y)|\ll 1$ for $y\in[0,1]$ we get
\begin{align}  \label{21_27}
\begin{split} 
\displaystyle\sum_{p\leq x}\dfrac{\log^2 p}{p}\,g^2_{n}\bigg(\dfrac{\log p}{\log x}\bigg)& =  \int_{2^-}^{x^{+}}g^2_{n}\left( \frac{\log u}{\log x}
\right)\d P(u)= \left(\int_{0}^{1}\,\alpha\,g^2_{n}(\alpha)\,\d \alpha\right)\log^2x + O(\log x).
\end{split}
\end{align}
Therefore, combining \eqref{21_26}, \eqref{20_13} and \eqref{21_27} in \eqref{eq:gPlusH1}, we conclude the proof of the proposition.
\end{proof}

\section{Computing $C_n$ numerically}\label{sec:num}
In this section we study the series that appears in the main term. For each $n\geq 1$, let $C_n$ be the series defined in \eqref{1_00}, i.e.
\[C_n = \sum_{m=2}^\infty \frac{\Lambda^2(m)}{m\left(\log m\right)^{2n+2}}.
\]
Then, Theorem \ref{thm:main} implies that
\[\int_0^T |S_n(t)|^2 \d t \sim \dfrac{C_n}{2\pi^2}T.
\]
Clearly, $C_n$ satisfies the estimates
\begin{equation*}
	\frac{1}{2(\log 2)^{2n}}\le C_n \le \frac{1}{2(\log 2)^{2n}} + \frac{A}{(\log 3)^{2n}},
\end{equation*}
for some universal constant $A>0$. Since $\log 2 <1$, then $C_n\rightarrow \infty$ as $n\rightarrow\infty$, with
\[C_n\sim \frac{1}{2(\log 2)^{2n}}.
\]
Let us obtain numerical bounds for $C_n$. To do this, we calculate numerically the first $x_n$ terms of the series and obtain explicit bounds for the tail 
$$V_n(x) = \sum_{m>x} \frac{\Lambda^2(m)}{m\left(\log m\right)^{2n+2}}.
$$
\begin{lemma} \label{boundforM}
Assume the Riemann Hypothesis. Define
\[M(x):= \sum_{m\le x} \Lambda^2(m).
\]
Then, for all $x\ge 10^5$, 
\begin{equation}\label{eq:M(x)}
        -0.047\sqrt{x}(\log x)^3 \le M(x) -(x\log x -x)  \le 0.057\sqrt{x}(\log x)^3,
 \end{equation}
 and
 \begin{equation}\label{eq:tail}
	    -\frac{0.017n+0.167}{\sqrt{x}(\log x)^{2n-1}} \le V_n(x) - \frac{1}{2n(\log x)^{2n}} \le 
	\frac{0.020n+0.181}{\sqrt{x}(\log x)^{2n-1}}.
	\end{equation}
\end{lemma}
\begin{proof}
We recall an explicit version of the Prime Number Theorem error term under RH (see \cite[Theorem 10]{Sch}): letting $\theta(x)=\sum_{p\le x}\log p$, for all $x\ge 600$ we have
\begin{equation*}
	\theta(x)=x+ O^*\bigg(\frac{\sqrt{x}\log^2x}{8\pi}\bigg).
\end{equation*}
We start by obtaining explicit bounds for $N(x):=\sum_{p\le x}\log^2 p$. Using integration by parts we have, for $x\ge 10^5$,
\begin{align*}
    N(x) = N(600) + \int_{600^+}^{x^+}\log y \,\d\theta(y)=  x\log x -x +c_0 +O^*\left(\frac{\sqrt{x}\log^3 x}{8\pi}\right) + O^*\bigg(\dfrac{1}{8\pi}\int_{600}^x\frac{\log^2y}{\sqrt{y}}\d y\bigg),
\end{align*}
where $c_0:= N(600) - \theta(600)\log600+600 = 62.9734...$ The above integral is bounded by
\begin{align*}
0\leq  \dfrac{1}{8\pi}\int_{600}^x\frac{\log^2y}{\sqrt{y}}\d y &\le \dfrac{\log^2x}{8\pi}\int_0^x \frac{\d y}{\sqrt{y}}\le \frac{\sqrt{x}\log^2x}{4\pi }\leq 0.00692\sqrt{x}\log^3x.
  \end{align*}
This gives
\begin{align}\label{eq:Nsimple2}
    N(x) &= x\log x -x + c_0 +O^*\big(0.04671\sqrt{x}\log^3x\big).
\end{align}
In particular, we obtain for $x\ge 10^5$ that $N(x)\le x\log x$. 
This inequality is also true for $45\leq x<10^5$ by numerical experiment. Now, using these estimates for $N(x)$, we obtain bounds for $M(x)$ as follows:
\begin{align*}
\begin{split}
    \sum_{p\le x}\log^2p \le \sum_{m\le x}\Lambda^2(m) &= \sum_{p\le x}\log^2p + \sum_{p^2\le x}\log^2p + \sum_{k=3}^{\llbracket\frac{\log x}{\log 2}\rrbracket}\sum_{p^k\le x}\log^2p\\
    &\le N(x) + N(\sqrt{x}) + \frac{\log x}{\log 2}N(\sqrt[3]{x})\\
    & \le x\log x -x +c_0 + 0.04671\sqrt{x}\log^3x+ 0.5\sqrt{x}\log x+ 0.4809\sqrt[3]{x}\log^2 x\\
    &\le  x\log x - x + 0.0568\sqrt{x}\log^3x,
\end{split}
\end{align*}
for $x\ge 10^5$. The lower bound follows from \eqref{eq:Nsimple2} and the fact that $c_0>0$. This proves \eqref{eq:M(x)}. Finally, let us prove \eqref{eq:tail}. We write $M(x) = x\log x -x +E(x)$. Then, integration by parts gives us
\begin{align} \label{17_33}
    V_n(x)=\int_{x^+}^{\infty} \dfrac{\d M(y)}{y(\log y)^{2n+2}}=\frac{1}{2n(\log x)^{2n}} - \frac{E(x)}{x(\log x)^{2n+2}} + \int_x^\infty \frac{E(y)\,(2n+2 +\log y)}{y^2(\log y)^{2n+3}}\d y.
\end{align}
Using the upper bound for $E(x)$ obtained in \eqref{eq:M(x)}, we have for $x\geq 10^5$,
\begin{align*} 
\int_x^\infty \frac{E(y)\,(2n+2 +\log y)}{y^2(\log y)^{2n+3}}\d y&\leq 0.057\int_x^\infty \frac{(2n+2 +\log y)}{y^{3/2}(\log y)^{2n}}\d y\\
&\leq\dfrac{0.057}{(\log x)^{2n}}\int_x^\infty \frac{(2n+2)}{y^{3/2}}\d y +\dfrac{0.057}{(\log x)^{2n-1}}\int_x^\infty \frac{1}{y^{3/2}}\d y \\
&\leq\frac{0.020\,n+0.134}{\sqrt{x}(\log x)^{2n-1}}.
\end{align*}
 Similarly, for the same integral we obtain the lower bound $(-0.017\,n-0.110)/\sqrt{x}(\log x)^{2n-1}$. Finally, combining these estimates with \eqref{eq:M(x)} in \eqref{17_33} we conclude \eqref{eq:tail}.
\end{proof}
\vspace{0.2cm}
Table $2$ gives the bounds for $C_n$, applying \eqref{eq:tail} for a specific value $x_n$, in the small cases $1\leq n \leq 10$. For $n\ge 11$, it can be verified that $C_n$ is essentially given by its exponentially-growing first term $\frac{1}{2(\log2)^{2n}}$, up to an error of at most $0.1$.

\begin{table}
\begin{center}
\begin{tabular}{ |c|c|c|c| } 
 \hline
 $n$ & $x_n$ & Lower bound for $C_n$ & Upper bound for $C_n$ \\ 
 \hline \hline
 \textbf{1} & $10^8$ & 1.5651238 & 1.5651260\\
 \hline
 \textbf{2} & $10^7$ & 2.46232872 & 2.46232876\\
 \hline
 \textbf{3} & $5\cdot 10^5$ & 4.72243168 & 4.72243169\\
 \hline
 \textbf{4} & $10^5$ & 9.55058572 & 9.55058573\\
 \hline
 \textbf{5} & $10^5$ & 19.6650658 & 19.6650659\\
 \hline
 \textbf{6} & $10^5$ & 40.7601579 & 40.7601580\\
 \hline
 \textbf{7} & $10^5$ & 84.6986707 & 84.6986708\\
 \hline
 \textbf{8} & $10^5$ & 176.175788 & 176.175789\\
 \hline
 \textbf{9} & $10^5$ & 366.593383 & 366.593384\\
 \hline
 \textbf{10} & $10^5$ & 762.938920 & 762.938921\\
 \hline \hline
\end{tabular}
		\vspace{0.2cm}
\caption{Upper and lower bounds for $C_n$, for $1 \leq n \leq 10$.}
\end{center}
\end{table}

\medskip

\section*{Acknowledgements}
We would like to thank Emanuel Carneiro and Kristian Seip for their insightful comments. We also thank the anonymous referee for the thorough review. A.C. was supported by Grant 275113 of the Research Council of Norway. O.Q-H. was supported by CNPq - Brazil.

\end{document}